\documentclass[10pt,a4paper]{amsart}
\usepackage{verbatim}

\usepackage[T1]{fontenc}
\usepackage{graphicx}
\usepackage{enumerate}
\usepackage{amsmath,amsfonts,amssymb}
\usepackage{color}
\usepackage{cite}
\definecolor{citation}{rgb}{0.2,0.58,0.2} 
\definecolor{formula}{rgb}{0.1,0.2,0.6}
\definecolor{url}{rgb}{0.3,0,0.5}
\usepackage{pgf,tikz}
\usepackage{mathrsfs}
\usepackage{fancyhdr}
\usepackage{dsfont}
\usepackage{ulem}
\usepackage{mathtools}

\title[Generalized superharmonic functions]{Generalized superharmonic functions\\ with strongly nonlinear operator}

\author{Iwona Chlebicka}\address{Iwona Chlebicka \\
Institute of Applied Mathematics and Mechanics, University of Warsaw \\ ul. Banacha 2, 02-097 Warsaw, Poland\\  \texttt{e-mail: i.chlebicka@mimuw.edu.pl}} 
\author{Anna Zatorska-Goldstein}\address{Anna Zatorska-Goldstein \\ Institute of Applied Mathematics and Mechanics, University of Warsaw \\ ul. Banacha 2, 02-097 Warsaw, Poland \\ \texttt{e-mail: azator@mimuw.edu.pl}}

\subjclass[2010]{35J60, 35J70\vspace{1mm}} 

\keywords{Superharmonic functions, Harnack's principle, Poisson modification, Minimum Principle, Liouville's Theorem, Potential theory\vspace{1mm}}
\date{}
\begin{document}
\maketitle \sloppy

\thispagestyle{empty}

\belowdisplayskip=18pt plus 6pt minus 12pt \abovedisplayskip=18pt
plus 6pt minus 12pt
\parskip 4pt plus 1pt
\parindent 0pt

\newcommand{\barint}{
         \rule[.036in]{.12in}{.009in}\kern-.16in
          \displaystyle\int  } 
\def\R{{\mathbb{R}}}
\def\rp{{[0,\infty)}}
\def\r{{\mathbb{R}}}
\def\n{{\mathbb{N}}}
\def\l{{\mathbf{l}}}
\def\bu{{\bar{u}}}
\def\bg{{\bar{g}}}
\def\bG{{\bar{G}}}
\def\ba{{\bar{a}}}
\def\bv{{\bar{v}}}
\def\bmu{{\bar{\mu}}}
\def\rn{{\mathbb{R}^{n}}}
\def\rN{{\mathbb{R}^{N}}} 

\newcommand{\snr}[1]{\lvert #1\rvert}
\newcommand{\nr}[1]{\lVert #1 \rVert}

\newtheorem{theo}{\bf Theorem} 
\newtheorem{coro}{\bf Corollary}[section]
\newtheorem{lem}[coro]{\bf Lemma}
\newtheorem{rem}[coro]{\bf Remark} 
\newtheorem{defi}[coro]{\bf Definition} 
\newtheorem{ex}[coro]{\bf Example} 
\newtheorem{fact}[coro]{\bf Fact} 
\newtheorem{prop}[coro]{\bf Proposition}

\newcommand{\dv}{{\rm div}}
\def\aI{\texttt{(a1)}}
\def\aII{\texttt{(a2)}}
\newcommand{\MBm}{{(M_B^-)}}
\newcommand{\MBmj}{{(M_{B_{r_j}}^-)}}
\newcommand{\aBi}{{a^B_{\rm i}}}
\newcommand{\opA}{{\mathcal{ A}}}
\newcommand{\wt}{\widetilde}
\newcommand{\ve}{\varepsilon}
\newcommand{\vp}{\varphi}
\newcommand{\vt}{\vartheta}
\newcommand{\gb}{{g_\bullet}}
\newcommand{\gbn}{{(\gb)_n}}
\newcommand{\vr}{\varrho}
\newcommand{\pa}{\partial}
\newcommand{\cW}{{\mathcal{W}}}
\newcommand{\supp}{{\rm supp}}
\newcommand{\miu}{{\min_{\partial B_k}u}}

\newcommand{\data}{\textit{\texttt{data}}}


\parindent 1em

\begin{abstract}
We study properties of $\opA$-harmonic and $\opA$-superharmonic functions involving an operator having generalized Orlicz growth. Our framework embraces reflexive Orlicz spaces, as well as natural variants of variable exponent and double-phase spaces. In particular, Harnack's Principle and Minimum Principle are provided for $\opA$-superharmonic functions and boundary Harnack inequality is proven for $\opA$-harmonic functions.
\end{abstract}


\section{Introduction}

The cornerstone of the classical potential theory is the Dirichlet problem for harmonic functions.  The focus of the nonlinear potential theory is similar, however, harmonic functions are replaced by $p$-harmonic functions, that is, continuous solutions to the $p$-Laplace equation $-\Delta_p u=-\dv(|Du|^{p-2}Du)=0$, $1<p<\infty$. There are known attempts to adapt the theory to the case when the exponent varies in space, that is $p=p(x)$ for $x\in\Omega$ or the growth is non-polynomial. Inspired by the significant attention paid lately to problems with strongly nonstandard and non-uniformly elliptic growth e.g.~\cite{IC-pocket,ChDF,comi,hht,m,r} we aim at developing basics of potential theory for problems with essentially broader class of operators embracing in one theory as special cases Orlicz, variable exponent and double-phase generalizations of $p$-Laplacian. To cover whole the mentioned range of general growth problems we employ the framework described in the monograph~\cite{hahabook}. Let us stress that unlike the classical studies~\cite{hekima,KiMa92} the operator we consider does {\em not} enjoy homogeneity of a~form $\opA(x,k\xi)=|k|^{p-2}k\opA(x,\xi)$. Consequently, our class of solutions is {\em not} invariant with respect to scalar multiplication. Moreover, we allow for operators whose ellipticity is allowed to vary dramatically in the space variable. What is more, we do {\em not} need to assume in the definition of $\opA$-superharmonic function that it is integrable with some positive power, which is typically imposed in the variable exponent case, cf. e.g.~\cite{hhklm,laluto}.

 We study fine properties of $\opA$-superharmonic functions defined by the Comparison Principle {with respect to} continuous solutions to $-\dv\opA(x,Du)=0$. {Here  $\opA:\Omega\times\rn\to\rn$ is assumed to have} generalized Orlicz growth expressed by the means of an inhomogeneous convex $\Phi$--functions $\vp:\Omega\times\rp\to\rp$ satisfying natural  non-degeneracy and balance conditions, see Section~\ref{sec:prelim} for details. In turn, the solutions belong to the Musielak-Orlicz-Sobolev space $W^{1,\vp(\cdot)}(\Omega)$ described carefully in the monograph~\cite{hahabook}. {The assumptions on the operator are summarized below and will be referred to as \textbf{(A)} throughout the paper}.

\subsection*{Assumption (A)} 
{We assume that} $\Omega \subset \rn$, $n\ge 2$, is an open bounded set. Let a vector field $\opA:\Omega\times\rn\to\rn$ be a  Caratheodory's function, that is  $x\mapsto \opA(x,\cdot)$ is measurable and $z\mapsto \opA(\cdot,z)$ is continuous.  Assume further that the following growth and coercivity assumptions hold true for almost all $x\in \Omega$ and all  $z\in \mathbb{R}^{n}\setminus \{0\}$:
\begin{flalign}\label{A}
\begin{cases}
\ \snr{\opA(x,z)} \le c_1^\opA\vp\left(x,\snr{z}\right)/|z|,\\
\ c_2^\opA {\vp\left(x,\snr{z} \right)} \le \opA(x,z)\cdot z
\end{cases}
\end{flalign} 
with absolute constants $c_1^\opA,c_2^\opA>0$ and some function $\vp:\Omega\times\rp\to\rp$ being measurable with respect to the first variable, convex with respect to the second one and   satisfying  (A0), (A1), (aInc)$_p$ and (aDec)$_q$ with some $1<p\leq q<\infty$. {The precise statement of these conditions is given in Section \ref{sec:prelim}.} We collect all parameters of the problem as $ \data=\data(p,q,c_1^\opA,c_2^\opA). $

 Moreover, let $\opA$ be monotone in the sense that for a.a. $x\in \Omega$ and any distinct $z_{1},z_{2}\in \mathbb{R}^{n}$ it holds that
\begin{flalign*}
0< \,\langle \opA(x,z_{1})-\opA(x,z_{2}),z_{1}-z_{2}\rangle.
\end{flalign*} 
We shall consider weak solutions, $\opA$-supersolutions, $\opA$-superharmonic, and $\opA$-harmonic functions related to  the problem\begin{equation}
\label{eq:main}-\dv\, \opA(x,Du)= 0 \quad\text{in }\ \Omega.
\end{equation}
For precise definitions see~Section~\ref{sec:sols}.

\subsection*{Special cases}
{Besides the $p$-Laplace operator case, corresponding to the choice of} $\vp(x,s)=s^{p},$ $1<p<\infty$, we cover by one approach a wide range of more degenerate operators. When we take $\vp(x,s)=s^{p(x)}$, {with} $p: \Omega \to \mathbb{R}$ {such that} $1<p^{-}_{\Omega} \leq p(x)\leq p^{+}_{\Omega} < \infty$ {and  satisfying} $\log$-H\"older condition (a special case of (A1)), {we render the so-called $p(x)$-Laplace equation} 
\begin{align*}
0=-\Delta_{p(x)} u=-\dv(\snr{Du}^{p(x)-2}Du).
 \end{align*} 
 Within the framework studied in~\cite{comi} solutions to double phase version of the $p$-Laplacian \[0=-\dv\, \opA(x,Du)=-\dv\left(\omega(x)\big(\snr{Du}^{p-2}+a(x)\snr{Du}^{q-2}\big)Du\right)\] are analysed with $1<p\leq q<\infty$, possibly vanishing weight $0\leq a\in C^{0,\alpha}(\Omega)$ and $q/p\leq 1+\alpha/n$  (a~special case of (A1); sharp for density of regular functions)  and with  a bounded, measurable, separated from zero weight $\omega$. We embrace also the borderline case between the double phase space and the variable exponent one, cf.~\cite{bacomi-st}. Namely, we consider solutions to
\[0=-\dv \opA(x,Du)=-\dv\left(\omega(x)\snr{Du}^{p-2}\big(1+a(x)\log({\rm e}+\snr{Du})\big)Du\right)\] with $1<p<\infty$, log-H\"older continuous $a$ and  a bounded, measurable, separated from zero weight $\omega$. Having an $N$-function $B\in\Delta_2\cap\nabla_2$, we can allow for problems with the leading part of the operator with growth driven by $\vp(x,s)=B(s)$ with an example of \[0=-\dv \, \opA(x,Du)=-\dv\left(\omega (x)\tfrac{B(\snr{Du})}{\snr{Du}^2}Du\right)\]
with a bounded, measurable, and separated from zero weight $\omega$. 
To give more new examples one can consider problems stated in weighted Orlicz (if $\vp(x,s)=a(x)B(s)$), variable exponent double phase (if $\vp(x,s)=s^{p(x)}+a(x)s^{q(x)}$), or multi phase Orlicz cases (if $\vp(x,s)=\sum_i a_i(x)B_i(s)$), as long as $\vp(x,s)$ is comparable to a~function doubling  with respect to the second variable and it satisfies the non-degeneracy and no-jump assumptions (A0)-(A1), see Section~\ref{sec:prelim}.

\subsection*{State of art} The key references for already classical nonlinear potential theory are~\cite{adams-hedberg,hekima,KiMa92}, but its foundations date back further to~\cite{HaMa1,HedWol}. A complete overview of the theory for equations with $p$-growth is presented in~\cite{KuMi2014}. The first generalization of~potential theory towards nonstandard growth is done in the weighted case~\cite{Mik,Tur}.  So far  significant attention was put on the variable exponent case, see e.g.~\cite{Alk,hhklm,hhlt,hklmp,laluto}, and analysis of related problems over metric spaces~\cite{BB}, there are some results obtained in the double-phase case~\cite{fz}, but to our best knowledge the Orlicz case is not yet covered by any comprehensive study stemming from~\cite{lieb,Maly-Orlicz}. 

Let us mention the recent advances within the theory. Supersolutions to~\eqref{eq:main} are in fact {solutions to} measure data problems with nonnegative measure, that enjoy lately the separate interest, cf.~\cite{ACCZG,IC-gradest,IC-measure-data,IC-lower,CGZG,CiMa,KiKuTu,KuMi2014,min-grad-est} concentrating on their existence and gradient estimates. The generalization of studies on removable sets for H\"older continuous solutions provided by~\cite{kizo} to the case of strongly non-uniformly elliptic operators has been carried out lately in~\cite{ChDF,ChKa}.  There are available various regularity results for related quasiminimizers having Orlicz or generalized Orlicz growth~\cite{hh-zaa,hhl,hht,hklmp,ka,kale,Maly-Orlicz}. For other recent developments in the understanding of the functional setting {we refer} also to~\cite{CUH,yags,haju,CGSGWK}.

 \subsection*{Applications} This kind of results are useful in getting potential estimates for solutions to measure data problems, entailing further regularity properties of their solutions, cf.~\cite{KiMa92,KuMi2014,KuMi2013}.  Particularly, the Maximum and Minimum Principles for $p$-harmonic functions together with properties of Poisson modification of $p$-superharmonic functions are important tools in getting Wolff potential estimates via the methods of~\cite{KoKu,tru-wa}. In fact, developing this approach further, {we employ} the results of our paper in the proof of Wolff potential estimates for problems with Orlicz growth~\cite{CGZG-Wolff}. They directly entail many natural and sharp regularity consequences and Orlicz version of the Hedberg--Wolff Theorem yielding full characterization of the  natural dual space to the space of solutions by the means of the  Wolff potential (see~\cite{HedWol} for the classical version).

\subsection*{Results and organization}  Section~\ref{sec:prelim} is devoted to notation and basic information on the setting. In Section~\ref{sec:sols} we define weak solutions, $\opA$-supersolutions, $\opA$-harmonic and $\opA$-superharmonic functions and provide proofs of their fundamental properties including the Harnack inequality for $\opA$-harmonic functions (Theorem~\ref{theo:Har-A-harm}). Further analysis of $\opA$-superharmonic functions is carried out in Section~\ref{sec:A-sh}. We prove there Harnack's Principle (Theorem~\ref{theo:harnack-principle}), fundamental properties of  Poisson's modification (Theorem~\ref{theo:Pois}), and Strong Minimum Principle (Theorem~\ref{theo:mini-princ}) together with their consequence of the boundary Harnack inequality (Theorem~\ref{theo:boundary-harnack}) for $\opA$-harmonic functions.

\section{Preliminaries}\label{sec:prelim}

\subsection{Notation}
In the following we shall adopt the customary convention of denoting by $c$ a constant that may vary from line to line. Sometimes to skip rewriting a constant, we use $\lesssim$. By $a\simeq b$, we mean $a\lesssim b$ and $b\lesssim a$. By $B_R$ we shall denote a ball usually skipping prescribing its center, when it is not important. Then by $cB_R=B_{cR}$ we mean a ball with the same center as $B_R$, but with rescaled radius $cR$. 
 With $U\subset \mathbb{R}^{n}$ being a~measurable set with finite and positive measure $\snr{U}>0$, and with $f\colon U\to \mathbb{R}^{k}$, $k\ge 1$ being a measurable map, by
\begin{flalign*}
\barint_{U}f(x) \, dx =\frac{1}{\snr{U}}\int_{U}f(x) \,dx
\end{flalign*}
we mean the integral average of $f$ over $U$. We make use of symmetric truncation on level $k>0$,  $T_k:\R\to\R$, defined as follows  \begin{equation*}T_k(s)=\left\{\begin{array}{ll}s & |s|\leq k,\\ k\frac{s}{|s|}& |s|\geq k. \end{array}\right.  \label{Tk}\end{equation*}

\subsection{Generalized Orlicz functions} We employ the formalism introduced in the monograph~\cite{hahabook}. Let us present the framework.

For $L\geq 1$ a real-valued function $f$ is $L$-almost increasing, if $Lf(s) \geq f(t)$ for $s > t$; $f$ is called $L$-almost decreasing if $Lf(s) \leq f(t)$ for $s > t$.

\begin{defi}  We say that $\vp:\Omega\times\rp\to[0,\infty]$ is a convex $\Phi$--function, and write $\vp\in\Phi_c(\Omega)$, if the following conditions hold:
\begin{itemize}
\item[(i)]  For every $s\in\rp$ the function $x\mapsto\vp(x, s)$ is
measurable and for a.e. $x\in\Omega$ the function $s\mapsto\vp(x, s)$ is increasing, convex, and left-continuous.
\item[(ii)] $\vp(x, 0) = \lim_{s\to 0^+} \vp(x, s) = 0$ and $\lim_{s\to \infty} \vp(x, s) =
\infty$ for a.e. $x\in\Omega$.
\end{itemize}
\end{defi}
\noindent Further, we say that $\vp\in\Phi_c(\Omega)$ satisfies\begin{itemize}
\item[(aInc)$_p$] if there exist $L\geq 1$ and $p >1$ such that $s\mapsto\vp(x, s)/s^p$ is $L$-almost increasing in $\rp$ for every $x\in\Omega$, 
\item[(aDec)$_q$] if there exist $L\geq 1$ and $q >1$ such that $s\mapsto\vp(x, s)/s^q$ is $L$-almost decreasing in $\rp$ for every $x\in\Omega$.
\end{itemize}
\noindent By $\vp^{-1}$ we denote the inverse of a convex $\Phi$-function $\vp$ {with respect to the second variable}, that is
\[
\vp^{-1}(x,\tau) := \inf\{s \ge 0 \,:\, \vp(x,s)\ge \tau\}.
\] 
We shall consider those $\vp\in\Phi_c(\Omega)$, which satisfy the following set of conditions.
\begin{itemize}
\item[(A0)] There exists $\beta_0\in (0, 1]$ such that $\vp(x, \beta_0) \leq 1$ and
$\vp(x, 1/\beta_0) \geq 1$ for all $x\in\Omega$.
\item[(A1)] There exists $\beta_1\in(0,1)$, such that for every ball $B$ with $|B|\leq 1$ it holds that 
\[\beta_1\vp^{-1}(x,s)\leq\vp^{-1}(y,s)\quad\text{for every $s\in [1,1/|B|]$ and a.e. $x,y\in B\cap\Omega$}.\]
\item[(A2)] For every {$s>0$} there exist $\beta_2\in(0,1]$ and $h\in L^1(\Omega)\cap L^\infty(\Omega)$, such that  \[\vp(x,\beta_2 r)\leq\vp(y,r) +h(x)+h(y)\quad\text{for a.e. $x,y\in \Omega$ whenever $\vp(y,r)\in[0,s]$}.\]
\end{itemize}
Condition (A0) is imposed in order to exclude degeneracy, while (A1) can be interpreted as local continuity. Fundamental role is played also by (A2) which imposes balance of the growth of $\vp$ with respect to its variables separately. 

\emph{The Young conjugate} of $\vp\in\Phi_c(\Omega)$  is the function $\wt\vp:\Omega\times\rp\to[0,\infty]$
defined as $  \wt \vp (x,s) = \sup\{r \cdot s - \vp(x,r):\ r \in \rp\}.$  Note that Young conjugation is involute, i.e. $\wt{(\wt\vp)}=\vp$. Moreover, if $\vp\in\Phi_c(\Omega)$, then $\wt\vp\in\Phi_c(\Omega)$. 
For $\vp\in\Phi_c(\Omega)$, the following inequality of Fenchel--Young type holds true
$$ 
rs\leq \vp(x,r)+\wt\vp(x,s).$$

We say that a function $\vp$ satisfies $\Delta_2$-condition (and write $\vp\in\Delta_2$) if there exists a~constant $c>0$, such that for every $s\geq 0$ it holds $\vp(x,2s)\leq c(\vp(x,s)+1)$. If $\wt\vp\in\Delta_2,$ we say that $\vp$ satisfies $\nabla_2$-condition and denote it by $\vp\in\nabla_2$. If  $\vp,\wt\vp\in\Delta_2$, then we call $\vp$ a doubling function. If $\vp\in \Phi_c(\Omega)$ satisfies {\rm (aInc)$_p$} and {\rm (aDec)$_q$}, then $\vp\simeq\psi_1$ with some  $\psi_1\in \Phi_c(\Omega)$ satisfying $\Delta_2$-condition and $\wt\vp\simeq \wt\psi_2$ with some $\wt\psi_2\in \Phi_c(\Omega)$ satisfying $\Delta_2$-condition, so we can assume that functions within our framework are doubling. Note that also $\psi_1\simeq\wt\psi_2$.

In fact, within our framework 
\begin{equation}
\label{doubl-star}\wt\vp\left(x, {\vp(x,s)}/{s}\right)\sim \vp(x,s) \quad\text{for a.e. }\ x\in\Omega\ \text{ and all }\ s>0
\end{equation} 
for some constants depending only on $p$ and $q$.

\subsection{Function spaces}\label{ssec:spaces} 

For a comprehensive study of these spaces we refer to \cite{hahabook}. We always deal with spaces generated by $\vp\in\Phi_c(\Omega)$ satisfying (aInc)$_p$, (aDec)$_q$, (A0), (A1), and (A2). For $f\in L^0(\Omega)$ we define {\em the modular} $\vr_{\vp(\cdot),\Omega}$   by
\begin{equation}
    \label{modular}
\vr_{\vp(\cdot),\Omega} (f)=\int_\Omega\vp (x, | f(x)|) dx.
\end{equation}
When it is clear from the context we omit assigning the domain.

\noindent {\em The Musielak--Orlicz
space} is defined as the set
\[L^{\vp(\cdot)} (\Omega)= \{f \in L^0(\Omega):\ \  \lim_{\lambda\to 0^+}\vr_{\vp(\cdot),\Omega}(\lambda f) = 0\}\]
endowed with the Luxemburg norm
\[\|f\|_{\vp(\cdot)}=\inf \left\{\lambda > 0 :\ \  \vr_{\vp(\cdot),\Omega} \left(\tfrac 1\lambda f\right) \leq 1\right\} .\]
For $\vp\in\Phi_c(\Omega)$, the space $L^{\vp(\cdot)}(\Omega)$ is a Banach space~\cite[Theorem~2.3.13]{hahabook}. Moreover, the following H\"older inequality holds true\begin{equation}
\label{in:Hold}\|fg\|_{L^1(\Omega)}\leq 2\|f\|_{L^{\vp(\cdot)}(\Omega)}\|g\|_{L^{\wt\vp(\cdot)}(\Omega)}.
\end{equation}
 We define {\em the Musielak-Orlicz-Sobolev space}  $W^{1,\vp(\cdot)}(\Omega)$  as follows
\begin{equation*} 
W^{1,\vp(\cdot)}(\Omega)=\big\{f\in W^{1,1}_{loc}(\Omega):\ \ f,|D f|\in L^{\vp(\cdot)}(\Omega)\big\},
\end{equation*}where $D$ stands for distributional derivative. The space is considered endowed with the norm
\[
\|f\|_{W^{1,\vp(\cdot)}(\Omega)}=\inf\big\{\lambda>0 :\ \   \vr_{\vp(\cdot),\Omega} \left(\tfrac 1\lambda f\right)+ \vr_{\vp(\cdot),\Omega} \left(\tfrac 1\lambda Df\right)\leq 1\big\}\,.
\]
By $W_0^{1,\vp(\cdot)}(\Omega)$ we denote a closure of $C_0^\infty(\Omega)$ under the above norm.

Because of the growth conditions $W^{1,\vp(\cdot)}(\Omega)$ is a separable and reflexive space. Moreover, smooth functions are dense there. 

\begin{rem}\cite{hahabook} If  $\vp\in \Phi_c(\Omega)$ satisfies {\rm (aInc)$_p$}, {\rm (aDec)$_q$}, (A0), (A1), (A2), {then} strong (norm) topology of $W^{1,\vp(\cdot)}(\Omega)$ coincides with the sequensional modular topology. Moreover, smooth functions are dense in this space in both topologies.
\end{rem}

Note that as a consequence of \cite[Lemma~2.1]{bbggpv} for every function $u$, such that $T_k(u)\in W^{1,\vp(\cdot)}(\Omega)$ for every $k>0$ (with $T_k$ given by~\eqref{Tk}) there exists a (unique) measurable function
$Z_u : \Omega \to \rn$ such that
\begin{equation}\label{gengrad}
D T_k(u) = \chi_{\{|u|<k\}} Z_u\quad \hbox{for
a.e. in $\Omega$ and for every $k > 0$.}
\end{equation}
With an abuse of~notation, we denote $Z_u$ simply by $D u$ and call it a {\it generalized gradient}.

{\subsection{The operator} Let us motivate that the growth and coercivity conditions from~\eqref{A} imply the expected proper definition of the operator involved in problem~\eqref{eq:main}. We notice that in our regime  the operator   $\mathfrak{A}_{\vp(\cdot)}$ defined as $$
\mathfrak{A}_{\vp(\cdot)} v := \opA(x,Dv)
$$ is well defined as $\ \mathfrak{A}_{\vp(\cdot)} : W^{1,\vp(\cdot)}_0(\Omega) \to (W^{1,\vp(\cdot)}_0(\Omega))'\ $
via  
\begin{flalign*}
\langle\mathfrak{A}_{\vp(\cdot)}v,w\rangle:=\int_{\Omega}\opA(x,Dv)\cdot Dw  \,dx\quad \text{for}\quad w\in C^{\infty}_{0}(\Omega),
\end{flalign*}
where $\langle \cdot, \cdot \rangle$ denotes dual pairing between reflexive Banach spaces $W^{1,\vp(\cdot)}(\Omega))$ and $(W^{1,\vp(\cdot)}(\Omega))'$.  Indeed, when $v\in W^{1,\vp(\cdot)}(\Omega)$ and $w\in C_0^\infty(\Omega)$, growth conditions~\eqref{A}, H\"older's inequality~\eqref{in:Hold}, equivalence~\eqref{doubl-star}, and Poincar\'e inequality~\cite[Theorem~6.2.8]{hahabook} justify that
\begin{flalign}
\nonumber\snr{\langle \mathfrak{A}_{\vp(\cdot)}v,w \rangle}\le &\, c\int_{\Omega}\frac{\vp(x,\snr{Dv})}{\snr{Dv}}\snr{Dw} \ dx \le c\left \| \frac{\vp(\cdot,\snr{Dv})}{\snr{Dv}}\right \|_{L^{\wt \vp(\cdot)}(\Omega)}\nr{Dw}_{L^{\vp(\cdot)}(\Omega)}\nonumber \\
\le &\, c\nr{Dv}_{L^{ \vp(\cdot)}(\Omega)}\nr{Dw}_{L^{\vp(\cdot)}(\Omega)}\le c\nr{w}_{W^{1,\vp(\cdot)}(\Omega)}.\label{op}
\end{flalign}
By density argument, the operator is well-defined on  $W^{1,\vp(\cdot)}_0(\Omega)$.}

\section{Various types of solutions and the notion of $\opA$-harmonicity} \label{sec:sols} 
All the problems are considered under Assumption {\bf (A)}.

\subsection{Definitions and basic remarks} $\ $

A \underline{continuous} function $u\in W^{1,\vp(\cdot)}_{loc}(\Omega)$ is {called} an {\em $\opA$-harmonic function} in an open set $\Omega$ if it is a (weak) solution to the equation  $-\dv\opA(x,Du)= 0$, {i.e.,
\begin{equation}
\label{eq:main:0}
\int_\Omega \opA(x,Du)\cdot D\phi\,dx= 0\quad\text{for all }\ \phi\in C^\infty_0(\Omega).
\end{equation}
}
Existence and uniqueness of $\opA$-harmonic functions is proven in \cite{ChKa}.
\begin{prop} \label{prop:ex-Ahf} Under {\rm Assumption {\bf (A)}} if $\Omega$ is bounded and $w\in W^{1,\vp(\cdot)}(\Omega)\cap C(\Omega)$, then there exists a unique solution $u\in W^{1,\vp(\cdot)}(\Omega)\cap C(\Omega)$ to problem
\begin{equation*}
\begin{cases}-\dv\, \opA(x,Du)= 0\quad\text{in }\ \Omega,\\
u-w\in W_0^{1,\vp(\cdot)}(\Omega).\end{cases}
\end{equation*}
Moreover, $u$ is locally bounded and for every $E\Subset\Omega$ we have \[\|u\|_{L^\infty(E)}\leq c(\data, \|Du\|_{ W^{1,\vp(\cdot)}(\Omega)}).\]
\end{prop}{}

We call a function $u\in W^{1,\vp(\cdot)}_{loc}(\Omega)$ a (weak) {\em $\opA$-supersolution} to~\eqref{eq:main:0} if~$-\dv\opA(x,Du)\geq 0$ weakly in $\Omega$, that is \begin{equation*}
\int_\Omega \opA(x,Du)\cdot D\phi\,dx\geq 0\quad\text{for all }\ 0\leq\phi\in C^\infty_0(\Omega)
\end{equation*}
and a (weak) {\em $\opA$-subsolution} if $-\dv\opA(x,Du)\leq 0$ weakly in $\Omega$, that is \begin{equation*}
\int_\Omega \opA(x,Du)\cdot D\phi\,dx\leq 0\quad\text{for all }\ 0\leq\phi\in C^\infty_0(\Omega).
\end{equation*}
By density of smooth functions we can use actually test functions from $W^{1,\vp(\cdot)}_0(\Omega)$.

The classes of {\em $\opA$-superharmonic} and {\em $\opA$-subharmonic} are defined by the Comparison Principle. 
{ \begin{defi}\label{def:A-sh}
We say that function $u$ is $\opA$-superharmonic if 
\begin{itemize}
\item[(i)]  $u$ is lower semicontinuous;
\item[(ii)] $u \not\equiv \infty$ in any component of $\Omega$;
\item[(iii)] for any $K\Subset\Omega$ and any $\opA$-harmonic $h\in C(\overline {K})$ in $K$, $u\geq h$ on $\partial K$ implies $u\geq h$ in $K$.
\end{itemize}
We say that an {\color{black} upper} semicontinuous function $u$ is $\opA$-subharmonic if $(-u)$ is $\opA$-superharmonic.
\end{defi}
}

The above definitions have the following direct consequences.

\begin{lem}\label{lem:A-arm-loc-bdd-below}
An $\opA$-superharmonic function $u$ is locally bounded from below.\\ An $\opA$-subharmonic function $u$ is locally bounded from above.
\end{lem}
\begin{lem}\label{lem:A-h-is-great}
If $u$ is $\opA$-harmonic, then it is $\opA$-supersolution, $\opA$-subsolution, $\opA$-superharmonic, and $\opA$-subharmonic.
\end{lem}
{By minor modification of the proof of \cite[Lemma 4.3]{ka} we get the following fact.  
\begin{lem}\label{lem:comp-princ} Let $u\in W^{1,\vp(\cdot)}(\Omega)$ be an $\opA$-supersolution to \eqref{eq:main:0}, and $v\in W^{1,\vp(\cdot)}(\Omega)$ be an $\opA$-subsolution to \eqref{eq:main:0}. If $\min(u-v) \in W^{1,\vp(\cdot)}_0(\Omega)$, then $u \geq {v}$  a.e. in $\Omega$.
\end{lem}
We have the following estimate for $\opA$-supersolutions.  \begin{lem}[Lemma~5.1,~\cite{ChKa}]\label{lem:A-supers-cacc}If $u\in W^{1,\vp(\cdot)}(\Omega)$ is a nonnegative $\opA$-supersolution, $B\Subset\Omega,$ and $\eta\in C^{1}_{0}(B)$ is such that $0\leq \eta\leq 1$. Then for all $\gamma\in(1,p)$ there holds
\begin{flalign*}
 \int_{B }u^{-\gamma}\eta^{q}\vp(x,\snr{D u}) \, dx\le c\int_{B }u^{-\gamma}\vp(x,\snr{D\eta}u) \, dx\quad\text{
 with $\ \ c=c(\data,\gamma)$.}
 \end{flalign*}
 \end{lem}}

It is well known that solutions, subsolutions, and supersolutions can be described by the theory of quasiminimizers. Since many of the results on quasiminizers from~\cite{hh-zaa} apply to our $\opA$-harmonic functions we shall recall the definition. 

Among all functions having the same `boundary datum' $w\in W^{1,\vp(\cdot)}(\Omega)$ the function $u\in W^{1,\vp(\cdot)}$ is a {\em quasiminimizer} if it has the least energy up to a factor $C$, that is if $(u-w)\in W_0^{1,\vp(\cdot)}(\Omega)$ and
\begin{equation}\label{def-quasiminimizer}
     \int_\Omega \vp(x,|D u|)\,dx\leq C\int_\Omega \vp(x, |D(u+v)|)\,dx
\end{equation}
holds true with an absolute constant $C>0$ for every $v\in W_0^{1,\vp(\cdot)}(\Omega)$. We call a~function $u$ {\em superquasiminimizer} ({\em subquasiminimizer}) if~\eqref{def-quasiminimizer} holds for all $v$ as above that are additionally nonnegative (nonpositive).

\begin{lem}\label{lem:Ah-is-quasi}
An $\opA$-harmonic function $u$  is a quasiminimizer.
\end{lem}{}

\begin{proof}
Let us take an arbitrary $v\in W_0^{1,\vp(\cdot)}(\Omega)$. We may write $v = w + \tilde{v} - u$ with `boundary datum' $w$ and any $\tilde{v} \in W_0^{1,\vp(\cdot)}(\Omega)$, and upon testing the equation \eqref{eq:main} with $v$ we obtain  
\[  \int_\Omega \opA(x,D u)\cdot Du\,dx=\int_\Omega \opA(x,D u)\cdot D(w+\tilde{v})\,dx.\]
Then by coercivity of $\opA$, Young's inequality, growth of $\opA$ and doubling growth of~$\vp$, for every $\ve>0$ we have
\begin{flalign*}
c_2^\opA \int_\Omega \vp(x,|D u|)\,dx&\leq \int_\Omega \opA(x,D u)\cdot Du\,dx=\int_\Omega \opA(x,D u)\cdot D(w+\tilde{v})\,dx\\
&\leq \ve \int_\Omega \wt\vp(x,|\opA(x,D u)|)\,dx+c(\ve)\int_\Omega \vp(x, |D(w+\tilde{v})|)\,dx\\
&\leq \ve \int_\Omega \wt\vp(x,c_1^\opA\vp(x,|D u|)/|Du|)\,dx+c(\ve)\int_\Omega \vp(x, |D(w+\tilde{v})|)\,dx\\
&\leq \ve \bar c \int_\Omega \vp(x,|D u|)\,dx+ c(\ve)\int_\Omega \vp(x, |D(w+\tilde{v})|)\,dx
\end{flalign*}
with $\bar c=\bar c(\data)>0.$ Let us choose $\ve>0$ small enough for the first term on the right-hand side can be absorbed on the left-hand side. By rearranging terms, and using the fact that $u+v=w+\tilde{v}$ we get that
\[ 
\int_\Omega \vp(x,|D u|)\,dx\leq C\int_\Omega \vp(x, |D(u+v)|)\,dx\quad \text{
with $\ \ C=C(\data)>0$}.
\]  
Hence we get the claim.
\end{proof}

\noindent By the same calculations as in the above proof we have the following corollary.
\begin{coro}
If $u$ is $\opA$-supersolution, then $u$ is a superquasiminizer, i.e.~\eqref{def-quasiminimizer} holds for all nonnegative $v\in W_0^{1,\vp(\cdot)}(\Omega)$.
\end{coro}

\subsection{Obstacle problem }

We consider the set
\begin{flalign}\label{con}
\mathcal{K}_{\psi,w}(\Omega):=\left\{ v\in W^{1,\vp(\cdot)}(\Omega)\colon \ v\ge \psi \ \  \mbox{a.e. in} \ \Omega \ \ \mbox{and} \ \ v-{w}\in W^{1,\vp(\cdot)}_{0}(\Omega)  \right\},
\end{flalign}
where we call $\psi:\Omega\to\overline{R}$ the obstacle and $w \in W^{1,\vp(\cdot)}(\Omega)$ the boundary datum. If $
\mathcal{K}_{\psi,w}(\Omega)\neq\emptyset$ by a~solution to the obstacle problem we mean a function $u\in \mathcal{K}_{\psi,w}(\Omega)$ satisfying
\begin{flalign}\label{obs}
\int_{\Omega}\opA(x,Du)\cdot D(v-u) \ dx \ge 0 \quad \mbox{for all } \ v\in \mathcal{K}_{\psi,w}(\Omega).
\end{flalign} 
We note the following basic information on the existence, the uniqueness, and the Comparison Principle for the obstacle problem are provided in \cite{kale} and~\cite[Section~4]{ChKa}.
\begin{prop}[Theorem 2, \cite{ChKa}]\label{prop:obst-ex-cont}
Under~{\rm  Assumption {\bf (A)}} let the obstacle $\psi\in W^{1,\vp(\cdot)}(\Omega)\cup\{-\infty\}$ and the boundary datum $w\in W^{1,\vp(\cdot)}(\Omega)$ be such that $\mathcal{K}_{\psi,w}(\Omega)\not =\emptyset$. Then there exists a function $u\in \mathcal{K}_{\psi,w}(\Omega)$ being a unique solution to the $\mathcal{K}_{\psi,w}(\Omega)$-obstacle problem \eqref{obs}. Moreover, if $\psi\in W^{1,\vp(\cdot)}(\Omega)\cap C(\Omega)$, then $v$ is continuous and is $\opA$-harmonic in the open set $\{x\in \Omega\colon u(x)>\psi(x)\}$.
\end{prop}
\noindent For more properties of solutions to related obstacle problems  see also~\cite{BCP,Obs1,ChDF,Obs2,hhklm,ka}. In particular, in~\cite{ka} several basic properties of quasiminimizers to related variational obstacle problem are proven. 

\begin{prop}[Proposition 4.3, \cite{ChKa}]
\label{prop:cacc}
Let $B_r \Subset B_R \subset \Omega$. 
Under assumptions of Proposition~\ref{prop:obst-ex-cont},
\begin{enumerate}
\item if  $u$ is a solution to the $\mathcal{K}_{\psi,w}(\Omega)$-obstacle problem \eqref{obs}, then there exists $c=c(\data,n)$, such that
\begin{align*}
\int_{ B_R} \vp(x, |D(u-k)_+ |) \, dx \leq c \int_{ B_R} \vp\left (x,\dfrac{(u-k)_+}{R-r}\right ) \, dx,\ \text{
where $k \geq \sup_{x \in B_R} \psi(x)$.}
\end{align*}
\item if  $u$ is a $\opA$-supersolution to \eqref{eq:main:0} in $\Omega$, then there exists $c=c(\data,n)$, such that
\begin{align*}
\int_{B_R} \vp(x, |Du_{-}|) \, dx \leq c \int_{ B_R} \vp\left (x,\dfrac{|u_{-}|}{R}\right ) \, dx.
\end{align*}
\end{enumerate}
\end{prop}
Note that in fact in~\cite[Proposition~4.3]{ka} only {\it (1)} is proven in detail, but {\it (2)} follows by the same arguments. 

\section{$\opA$-superharmonic functions}\label{sec:A-sh}
\subsection{Basic observations}

\begin{prop}[Comparison Principle]\label{prop:comp-princ} Suppose $u$ is $\opA$-superharmonic and $v$ is $\opA$-subharmonic in $\Omega$.
If $\limsup_{y\to x} v(y)\leq\liminf_{y\to x} u(y)$
for all $x\in\partial \Omega$ (excluding the cases $-\infty\leq-\infty$ and $\infty\leq \infty$), then $v\leq u$ in $\Omega$.
\end{prop}
\begin{proof} When we fix $x\in\Omega$ and $\ve>0$, by the assumption we can find a regular open set $D\Subset\Omega,$ such that $v<u+\ve$ on $\partial D.$ Pick a decreasing sequence $\{\phi_k\}\subset C^\infty(\Omega)$ converging to $v$ pointwise in $\overline{D}$. Since $\partial D$ is compact by lower semicontinuity of $(u+\ve)$ we infer that $\phi_k\leq u+\ve$ on $\partial D$ for some $k$. We take a function $h$ being $\opA$-harmonic in $D$ coinciding with $\phi_k$ on $\partial D$. By definition it is continuous
up to a boundary of $D$. Therefore, $v\leq h\leq u+\ve$ on $\partial D$ and so $v\leq h\leq u+\ve$ in $D$ as well. We get the claim by letting $\ve\to 0$. \end{proof}{}

\begin{coro}\label{coro:min-A-super}Having the Comparison Principle one can deduce what follows. 
\begin{itemize}
\item[(i)] If $a_1,a_2\in\R$, $a_1\geq 0$, and $u$ is $\opA$-superharmonic in $\Omega,$ then so is $a_1u+a_2$.
    \item[(ii)] If $u$ and $v$ are $\opA$-superharmonic in $\Omega,$ then so is $\min\{u,v\}.$
    \item[(iii)] Suppose $u$ is not identically $\infty$, then $u$ is $\opA$-superharmonic in $\Omega$ if and only if $\min\{u,k\}$ is $\opA$-superharmonic in $\Omega$ for every $k=1,2,\dots$.
    \item[(iv)] The function $u$ is $\opA$-superharmonic in $\Omega$, if it is $\opA$-superharmonic in every component of $\Omega.$ 
    \item[(v)] If $u$ is $\opA$-superharmonic and finite a.e. in $\Omega$ and $E\subset\Omega$ is a nonempty open subset, then $u$ is $\opA$-superharmonic in $E$.
\end{itemize}{} 
\end{coro}

\begin{lem}\label{lem:pasting}
Suppose $D\subset\Omega$, $u$ is $\opA$-superharmonic in $\Omega$, and  $v$ is $\opA$-superharmonic in $D$. If the function
\[w=\begin{cases}
\min\{u,v\}\quad&\text{in }\ D,\\
u\quad&\text{in }\ \Omega\setminus D
\end{cases}{}\]
is lower semicontinuous, then it is $\opA$-superharmonic in $\Omega$.
\end{lem}
\begin{proof}Let $E\Subset\Omega$ be open and $h$ be an $\opA$-harmonic function, such that $h\leq w$ on $\partial E.$ By the Comparison Principle of Proposition~\ref{prop:comp-princ} we infer that $h\leq w$ in $\overline{E}$. Since $w$ 
is lower semicontinuous, for every $x\in\partial D\cap E$ it holds that 
\[
\lim_{\substack{y\in D\cap\Omega \\ y\to x}}h(y)\leq u(x)=w(x)\leq\liminf_{\substack{y\in D\cap\Omega \\ y\to x}} v(y).
\]
Consequently, for every $x\in\partial (D\cap E)$ one has \[
\lim_{\substack{y\in D\cap\Omega \\ y\to x}}h(y)\leq w(x)\leq\liminf_{\substack{y\in D\cap\Omega \\ y\to x}}w(y).\]
By the Comparison Principle of Proposition~\ref{prop:comp-princ} also $h\leq w$ in $D\cap E$. Then $h\leq w$ in $E$, what was to prove.
\end{proof}{}{}

\begin{lem}\label{lem:cont-supersol-are-superharm} If $u$ is a continuous $\opA$-supersolution, then it is $\opA$-superharmonic.
\end{lem}
\begin{proof}
Since $u$ is continuous and finite a.e. (because it belongs to $W^{1,\vp(\cdot)}_{loc}(\Omega)$), we have to prove only that Comparison Principle for $\opA$-superharmonic functions holds.

Let $G \Subset \Omega$ be an open set, and let $h$ be a continuous, $\opA$-harmonic function in $G$, such that $h \leq u$ on $\partial G$. Fix $\epsilon >0$ and choose and open set $E \Subset G$ such that $u + \epsilon \geq h$ in $G \setminus E$. Since the function $\min\{u+\epsilon-h,0\}$ has compact support, it belongs to $W^{1,\vp(\cdot)}(E)$. Hence Lemma \ref{lem:comp-princ} implies $u+\epsilon \geq h$ in $E$, and therefore a.e. in $G$. Since the function is continuous, the inequality is true in each point of $G$. As $\epsilon$ was chosen arbitrary, the claim follows.\end{proof}

We shall prove that $\opA$-superharmonic functions can be  approximated from below by $\opA$-supersolutions.

\begin{prop}
\label{prop:from-below}
Let $u$ be $\opA$-superharmonic in $\Omega$ and let $G\Subset\Omega$. Then there exists a nondecreasing sequence of continuous $\opA$-supersolutions $\{u_j\}$ in $G$ such that $u=\lim _{j\to\infty}u_j$ pointwise in $G$. For nonnegative $u$, approximate functions $u_j$ can be chosen nonnegative as well. 
\end{prop}
\begin{proof} 
Since $u$ is lower semicontinuous in $\overline{G}$, it is bounded from below and there exists a nondecreasing sequence $\{\phi_j\}$ of Lipschitz functions on $\overline{G}$ such that $u=\lim _{j\to\infty}\phi_j$ in $G$. For nonnegative $u$, obviously $\phi_j,$ $j\in\mathbb{N}$ can be chosen nonnegative as well. Let $u_j$ be the {solution of the $\mathcal{K}_{\phi_j,\phi_j}(G)$-obstacle problem which by Proposition~\ref{prop:obst-ex-cont} is continuous} and \[\phi_j<u_j  \qquad\text{in the open set }\ A_j=\{x\in G:\ \phi_j\neq u_j\}.\]
Moreover, $u_j$ is $\opA$-harmonic in $A_j.$ By Comparison Principle from Proposition~\ref{prop:comp-princ} we infer that the sequence $\{u_j\}$ is nondecreasing. Since $u$ is $\opA$-superharmonic, we have $u_j\leq u$ in $A_j$. Then consequently $\phi_j\leq u_j\leq u$ in $ G.$ Passing to the limit with $j\to\infty$ we get that $u=\lim _{j\to\infty}u_j$, what completes the proof. 
\end{proof}
\begin{lem}\label{lem:loc-bdd-superharm-are-supersol}
If $u$ is $\opA$-superharmonic in $\Omega$ and locally bounded from above, then $u\in W^{1,\vp(\cdot)}_{loc}(\Omega)$ and $u$ is  $\opA$-supersolution in $\Omega$.
\end{lem}
\begin{proof}
Fix open sets $E\Subset G \Subset \Omega$. By Proposition \ref{prop:from-below} there exists a nondecreasing sequence of continuous $\opA$-supersolutions $\{u_j\}$ in $G$ such that $u=\lim _{j\to\infty}u_j$ pointwise in $G$. Since $u$ is locally bounded we may assume  {$u_j\leq u <0$} in $G$. It follows from Proposition \ref{prop:cacc} that the sequence {$\{|Du_j |\}$} is locally bounded in $L^{\vp(\cdot)}(G)$. Since $u_j \to u$ a.e. in $G$, it follows that $u \in W^{1,\vp(\cdot)}(G)$, and $Du_j \rightharpoonup Du$ weakly in $L^{\vp(\cdot)}(G)$. 

We need to show now that $u$ is an $\opA$-supersolution in $\Omega$. To this end we first prove that (up to a subsequence) gradients {$\{Du_j \}$} converge a.e. in $G$.   We start with proving that 
\begin{equation}
\label{Ijto0}
I_j = \int_{E} \Big(\opA(x,Du) - \opA(x,Du_j) \Big)\cdot \big( Du - Du_j \big)\, dx\to 0 \quad \text{as}\ j\to\infty.
\end{equation}
Choose $\eta \in C_0^\infty(G)$ such that $0 \leq \eta \leq 1$, and $\eta = 1$ in {$E$}. Using $\psi = \eta(u-u_j)$ as a test function for the $\opA$-supersolution $u_j$ and applying the H\"older inequality, the doubling property of $\vp$, and the Lebesgue dominated monotone convergence theorem we obtain
\begin{align*}
-\int_G \eta \opA(x,Du_j) &\cdot \big( Du - Du_j \big)\, dx 
\leq
\int_G (u-u_j)\opA(x,Du_j)\cdot D\eta\, dx \\
&\leq 2 \|(u-u_j)D\eta\|_{L^{\vp(\cdot)}(G)} \|\opA(\cdot,Du_j) \|_{L^{\wt\vp(\cdot)}(G)} \\
&\leq c \|u-u_j\|_{L^{\vp(\cdot)}(G)} \to 0.
\end{align*}
Moreover, since
$$
\eta \opA(\cdot , Du) \in L^{\wt\vp(\cdot)}(G),
$$
the weak convergence  $Du_j \rightharpoonup Du$ in $L^{\vp(\cdot)}(G)$ implies
$$
\int_G \eta \opA(x,Du)\cdot \big( Du - Du_j \big)\, dx \to 0.
$$
Then, since $\eta \big(\opA(x,Du) - \opA(x,Du_j) \big)\cdot \big( Du - Du_j \big) \geq 0$ a.e. in $G$, we conclude with~\eqref{Ijto0}. Since the integrand in $I_j$ is nonnegative, we may pick up a subsequence (still denoted $u_j$) such that 
\begin{equation} \label{eq:point-conv}
\Big(\opA(x,Du(x)) - \opA(x,Du_j(x)) \Big)\cdot \big( Du(x) - Du_j(x) \big) \to 0\ \ \text{
for a.a. $x\in E$.}
\end{equation}
Fix $x \in E$ such that \eqref{eq:point-conv} is valid, and that $|Du(x)| < \infty$. Upon choosing further subsequence we may assume that\[{Du_j(x)}\to \xi \in \overline{\R^n}.\] Since we have
\begin{align*}
\big(\opA(x,&Du(x)) - \opA(x,Du_j(x)) \big)\cdot \big( Du(x) - Du_j(x) \big) \\
&\geq
c_2^\opA \vp(x,|Du_j(x)|) - c_1^\opA \frac{\vp(x, |Du(x)|)}{|Du(x)|} |Du_j(x)| - c_1^\opA \frac{\vp(x, |Du_j(x)|)}{|Du_j(x)|} |Du(x)| \\
&\geq c(\data,|Du(x)|) \vp(x,|Du_j(x)|) \left(1- \frac{|Du_j(x)|}{\vp(x,|Du_j(x)|)} - \frac{1}{|Du_j(x)|} \right)
\end{align*}
and \eqref{eq:point-conv} is true,  it must follow that $|\xi| < \infty$.

Since the mapping $\zeta \mapsto \opA(x, \zeta)$ is continuous, we have
$$
\big(\opA(x,Du(x)) - \opA(x,\xi) \big)\cdot \big( Du(x) - \xi \big) = 0
$$
and it follows that $\xi = Du(x)$, and
$$
Du_j(x) \to Du(x) \qquad \text{for a.e. $x \in E$},
$$
and
$$
\opA(\cdot, Du_j) \rightharpoonup \opA(\cdot, Du) \qquad \text{weakly in $L^{\wt\vp(\cdot)}$}.
$$
Therefore that $u$ is an $\opA$-supersolution of \eqref{eq:main:0}. Indeed, if $\phi \in C_0^\infty(\Omega),$ $\phi \geq 0$ is such that $\supp\, \phi \subset E$, then {$D\phi \in L^{\vp(\cdot)}(E)$} and we have
\begin{align*}
0 \leq \int_\Omega \opA(x, Du_j) \cdot D\phi\, dx \to  
\int_\Omega \opA(x, Du) \cdot D\phi\, dx \quad \text{as}\ j\to\infty.
\end{align*}
Since $E$ was arbitrary this concludes the proof.
\end{proof}

\subsection{Harnack's inequalities}

In order to get strong Harnack's inequality for $\opA$-harmonic function and weak Harnack's inequality for $\opA$-superharmonic functions we need related estimates proved for $\opA$-subsolutions and $\opA$-supersolutions. Having Lemma~\ref{lem:Ah-is-quasi} we can specify results derived for quasiminizers in~\cite{hh-zaa} to our case.

\begin{prop}[Corollary~3.6, \cite{hh-zaa}]
\label{prop:weak-Har-sub-sup}
For a locally bounded  function $u\in W^{1,\vp(\cdot)}_{loc}(\Omega)$ being $\opA$-subsolution in $\Omega$ there exist constants $R_0=R_0(n)>0$ and $C=C(\data,n,R_0,{\rm ess\,sup}_{B_{R_0}} u)>0$, such that
\[{\rm ess\,sup}_{B_{R/2}}u-k\leq C\left(\left(\barint_{B_R}(u-k)_+^{s}\,dx\right)^\frac{1}{s}+R\right) \]
for all $R\in(0,R_0]$, $s>0$ and $k\in \R$.
\end{prop}

\begin{prop}[Theorem~4.3, \cite{hh-zaa}]
\label{prop:weak-Har-super-inf}
For a nonnegative function $u\in W^{1,\vp(\cdot)}_{loc}(\Omega)$ $\opA$-supersolution in $\Omega$ there exist constants $R_0=R(n)>0$, $s_0=s_0(\data,n)>0$ and $C=C(\data,n)>0$, such that
\[  \left(\barint_{B_R}u^{s_0}\,dx\right)^\frac{1}{s_0} \leq C\left({\rm ess\,inf}_{B_{R/2}} u+R\right) \]
for all $R\in(0,R_0]$ provided $B_{3R}\Subset\Omega$ and $\vr_{\vp(\cdot),B_{3R}}(Du)\leq 1.$
\end{prop}
Let us comment on the above result. For the application in \cite{hh-zaa} dependency of $s_0$ on other parameters is not important and so -- not studied with attention. Actually, this theorem is not proven in detail in \cite{hh-zaa}, but refers to standard arguments presented in~\cite{hht,hklmp}. Their re-verification enables to find $s_0=s_0(\data,n)$. Let us note that after we completed our manuscript, an interesting study on the weak Harnack inequalities with an explicit exponent, holding for unbounded supersolutions, within our framework of generalized Orlicz spaces appeared, see~\cite{bhhk}.

{Since $\opA$-harmonic function is an $\opA$-subsolution and and $\opA$-supersolution at the same time (Lemma~\ref{lem:A-h-is-great}), by Propositions~\ref{prop:weak-Har-sub-sup} and~\ref{prop:weak-Har-super-inf} we infer the full Harnack inequality.}
\begin{theo}[Harnack's inequality for $\opA$-harmonic functions]
\label{theo:Har-A-harm} For a nonnegative $\opA$-harmonic function $u\in W^{1,\vp(\cdot)}_{loc}(\Omega)$ there exist constants $R_0=R(n)>0$, $s_0=s_0(\data,n)>0$ and $C=C(\data,n,R_0,{\rm ess\,sup}_{B_{R_0}} u)>0$, such that
\[ {\rm ess\,sup}_{B_{R}}u \leq C\left({\rm ess\,inf}_{B_{R}} u+R\right) \]
for all $R\in(0,R_0]$ provided $B_{3R}\Subset\Omega$ and $\vr_{\vp(\cdot),B_{3R}}(Du)\leq 1.$
\end{theo} 


\subsection{Harnack's Principle for $\opA$-superharmonic functions}

We are going to characterize the limit of nondecreasing sequence of $\opA$-superharmonic functions and their gradients.

\begin{theo}[Harnack's Principle for $\opA$-superharmonic functions] \label{theo:harnack-principle}  Suppose that $u_i$, $i=1,2,\ldots$, are $\opA$-superharmonic and finite a.e. in $\Omega$. If the sequence $\{u_i\}$ is nondecreasing then the limit function $u=\lim_{i \to \infty} u_i$ {is  $\opA$-superharmonic or infinite in  $\Omega$.}  Furthermore, if $u_i$, $i=1,2,\ldots$, are nonnegative, then up to a subsequence also $Du_i\to Du$ a.e. in $\{u<\infty\},$ where `$D$' stands for the generalized gradient, cf.~\eqref{gengrad}.
\end{theo}{}
\begin{proof}  The proof is presented in three steps. We start with motivating that the limit function is either $\opA$-superharmonic or $u \equiv \infty$, then we concentrate on gradients initially proving the claim  for a priori globally bounded sequence $\{u_i\}$ and conclude by passing to the limit with the bound.
 
{\em Step 1.}  Since $u_i$ are lower semicontinuous, so is $u$.  The following fact holds: Given a compact set $K \Subset \Omega$, if $h \in C(K)$, $\epsilon >0$ is a small fixed number, and $u > h-\epsilon$ on $K$, then, for $i$ sufficiently large, $u_i > h-\epsilon$. Indeed, let's argue by contradiction. Assume that for every $i$ there exists $x_i\in K,$ such that
$$ u_i(x_i) \leq h(x_i) - \epsilon.
$$
Since $K$ is compact, we can assume that $x_i \to x_o$. Fix $l \in \mathbb{N}$. Then, for $i > l$ we have
$$
u_l(x_i) \leq u_i(x_i) \leq h(x_i) - \epsilon
$$
The right-hand side in the previous display tends with $i \to \infty$ to $h(x_o)-\epsilon$. Hence
$$
u_l(x_o) \leq \liminf_{i\to\infty} u_l(x_i) \leq h(x_o) - \epsilon
$$
Thus for every $l$ we have $u_l(x_o) \leq h(x_o) - \epsilon$, which implies $u(x_o) \leq h(x_o) - \epsilon$ which is in the contradiction with the fact that $u > h-\epsilon$ on $K$.

Using this fact we can prove that the limit function $u=\lim_{i \to \infty} u_i$ {is  $\opA$-superharmonic unless $u \equiv \infty$}.  Choose an open $\Omega' \Subset {\Omega}$ and $h \in C(\overline{\Omega'})$ an $\opA$-harmonic function. Assume the inequality $ u \geq h$ holds on $\partial \Omega'$. It follows that for every $\epsilon >0$ on $\partial \Omega'$ we have $u > h-\epsilon$  and, from the aforementioned fact, it follows that $u_i > h- \epsilon$ on $\partial \Omega'$. Since all $u_i$ are $\opA$-superharmonic, Proposition~\ref{prop:comp-princ} yields that $u_i \geq h-\epsilon$ on $\Omega'$. Therefore $u \geq h-\epsilon$ on $\Omega'$. Since $\epsilon $ is arbitrary, we have $u \geq h$ on $\Omega'$. Therefore the Comparison Principle from definition of $\opA$-superharmonic holds unless $u \equiv \infty$ in~$\Omega$. Finally, $u=\lim_{i \to \infty} u_i$ is $\opA$-superharmonic unless $u \equiv \infty$.

{\em Step 2.}  Assume $0\leq u_i\leq k$ for all $i$ with $k>1$ and choose open sets $E\Subset G\Subset{\Omega}$.  By Lemma~\ref{lem:A-supers-cacc} we get that
\[\vr_{\vp(\cdot),G}(Du_i)\leq c k^q\]
with $c=c(\data,n)>0$ uniform with respect to $i$. Then, by doubling properties of $\vp$, we infer that\begin{equation}\label{Duibound}
    \|Du_i\|_{L^{\vp(\cdot)}(G)}\leq c (\data,n,k).
\end{equation}
Consequently $\{u_i\}$ is bounded in $W^{1,\vp(\cdot)}(G)$ and $u_i\to u$ weakly in $W^{1,\vp(\cdot)}(G).$ Further, it has a non-relabelled subsequence converging a.e.~in $G$ to $u\in W^{1,\vp(\cdot)}(G)$.   Let us show that
\begin{equation}
    \label{grad=grad} Du_j\to Du\qquad\text{a.e. in }\ E.
\end{equation}{}
We fix arbitrary $\ve\in(0,1)$, denote \[J_i=\{x\in E:\ \big(\opA(x,Du_i(x))-\opA(x,Du(x))\big)\cdot(Du_i(x)-Du(x))>\ve\}\]
and estimate its measure. We have
\begin{flalign}
|J_i|\leq& |J_i\cap\{|u_i-u|\geq \ve^2\}|\nonumber\\
&+\frac{1}{\ve}\int_{J_i\cap\{|u_i-u|<\ve^2\}} \big(\opA(x,Du_i)-\opA(x,Du)\big)\cdot(Du_i-Du)\,dx.\label{Jiest1}
\end{flalign}{}
Let $\eta\in C_0^\infty(G)$ be such that $\mathds{1}_{E}\leq \eta\leq \mathds{1}_{G}$. We define
\[w_1^i=\min\big\{(u_i+\ve^2-u)^+,2\ve^2\big\}\quad\text{and}\quad w_2^i=\min\big\{(u+\ve^2-u_i)^+,2\ve^2\big\}.\]
Then $w_1^i\eta$ and $w_2^i\eta$ are nonnegative functions from $W^{1,\vp(\cdot)}_0(G)$ and can be used as test functions. Since $u$ and $u_i$, $i=1,2,\ldots$, are $\opA$-supersolutions we already know that $u_i\to u$ weakly in $W^{1,\vp(\cdot)}(E').$ By growth condition {we can estimate like in~\eqref{op}} and by~\eqref{Duibound} we have
\begin{flalign*}
\int_{G\cap\{|u_i-u|<\ve^2\}} \opA(x,Du)\cdot(Du_i-Du)\eta\,dx&\leq \int_{G\cap\{|u_i-u|<\ve^2\}} \opA(x,Du)\cdot D\eta\,w^i_1\,dx\\
&\leq \, c\ve^2\int_{G} \frac{\vp(x,|Du|)}{|Du|}|D\eta|\,dx\\
&\leq \, c\ve^2 
\end{flalign*}{}
with $c>0$ independent of $i$ and $\ve$. Analogously
\begin{flalign*}
\int_{G\cap\{|u_i-u|<\ve^2\}} \opA(x,Du_i)\cdot(Du_i-Du)\eta\,dx&\leq \, c\ve^2, 
\end{flalign*}{}

Summing up the above observations we have\begin{flalign*}\frac{1}{\ve}\int_{J_i\cap\{|u_i-u|<\ve^2\}} \big(\opA(x,Du_i)-\opA(x,Du)\big)\cdot(Du_i-Du)\,dx\leq c\ve.
\end{flalign*}{}
The left-hand side is nonnegative by the monotonicity of the operator, so due to~\eqref{Jiest1} we have \begin{flalign*}
|J_i|\leq& |J_i\cap\{|u_i-u|\geq \ve^2\}|+c\ve
\end{flalign*}{}
with $c>0$ independent of $i$ and $\ve$. By letting $\ve\to 0$ we get that $|E_j|\to 0$. Because of the strict monotonicity of the operator, we infer~\eqref{grad=grad}. {We can conclude the proof of this step by choosing a diagonal subsequence.}

{\em Step 3.} Now we concentrate on the general case. For every $k=1,2,\dots$ we select subsequences $\{u_{i}^{(k)}\}_k$ of $\{u_i\}$ and find an $\opA$-superharmonic function $v_k,$ such that $\{u_{i}^{(k+1)}\}\subset\{u_{i}^{(k)}\}$, $T_k(u^{(k)}_j)\to v_k$ and $D(T_k(u^{(k)}_j))\to D v_k$ a.e. in $\Omega$. We note that $v_k$ increases to a function, which is $\opA$-harmonic or equivalently infinite. Additionally, $v_k=T_k(u).$ The diagonally chosen subsequence $\{u_i^{(i)}\}$ has all the desired properties.
\end{proof}{} 

We have the following consequence of the Comparison Principle and Theorem~\ref{theo:harnack-principle}.
\begin{coro}[Harnack's Principle for $\opA$-harmonic functions]\label{coro:Ah-harnack-principle}  Suppose that $u_i$, $i=1,2,\ldots$, are $\opA$-harmonic in $\Omega$. If the sequence $\{u_i\}$ is nondecreasing then the limit function $u=\lim_{i \to \infty} u_i$ {is  $\opA$-harmonic  or infinite in $\Omega$.} 
\end{coro}

\subsection{Poisson modification}

The Poisson modification of an $\opA$-superharmonic function in a regular set $E$ carries the idea of its local smoothing. A boundary point is called regular if at this point the boundary value of any Musielak-Orlicz-Sobolev function is attained not only in the Sobolev sense but also pointwise.   A set is called regular if all of its boundary points are regular. See~\cite{hh-zaa} for the  result that if the complement of $\Omega$ is locally fat at $x_0\in\partial\Omega$ in the capacity sense, then $x_0$ is regular. Thereby of course polyhedra and balls are obviously regular.

Let us consider a function $u$, which is $\opA$-superharmonic {and finite a.e.} in $\Omega$ and an open set $E\Subset\Omega$ with regular $\overline{E}.$ We define
\[u_E=\inf\{v:\ v \ \text{is $\opA$-superharmonic in $E$ and }\liminf_{y\to x}v(y)\geq u(x)\ \text{for each }x\in\partial \overline E\}\]
and the {\em Poisson modification} of $u$ in $E$ by
\[P(u,E)=\begin{cases}
u\quad&\text{in }\ \Omega\setminus E,\\
u_E &\text{in }\  E.
\end{cases}{}\]

\begin{theo}[Fundamental properties of the Poisson modification]\label{theo:Pois} If $u$ is $\opA$-superharmonic {and finite a.e.} in $\Omega$, then its Poisson modification $P(u,E)$ is
\begin{itemize}
    \item [(i)] $\opA$-superharmonic in~$\Omega$,  
    \item [(ii)]  $\opA$-harmonic in $E$,
    \item [(iii)] $P(u,E)\leq u$ in~$\Omega.$
\end{itemize}
\end{theo}{}
\begin{proof} The fact that $P(u,E)\leq u$ in $\Omega$ results directly from the definition. By assumption $u$ is finite somewhere. Let us pick a nondecreasing sequence $\{\phi_i\}\subset C^\infty(\rn)$ which converges to $u$ in $\overline{E}$. Let $h_i$ be the unique $\opA$-harmonic function agreeing with $\phi_i$ on $\partial E.$ The sequence $\{h_i\}$ is nondecreasing by the Comparison Principle from Proposition~\ref{prop:comp-princ}. Since $h_i\leq u$, by Harnack's Principle from Corollary~\ref{coro:Ah-harnack-principle} we infer that \[h:=\lim_{i\to\infty}h_i\] 
is $\opA$-harmonic in $E$. Moreover, $h\leq u$ and thus $h$ is also finite somewhere. Since 
\[u(y)=\lim_{i\to\infty}\phi_i(y)\leq\liminf_{x\to y} h(x)\quad\text{for }\ y\in\partial E,\]
it follows that $P(u,E)\leq h$ in $E$. On the other hand, by the Comparison Principle (Proposition~\ref{prop:comp-princ}) we get that $h_i\leq P(u,E)$ in $E$ for every $i$. Therefore $P(u,E)|_E=h$ is $\opA$-harmonic in $E$. This reasoning also shows that $P(u,E)$ is lower semicontinuous and, by Lemma~\ref{lem:pasting}, it is also $\opA$-superharmonic in $\Omega$.
\end{proof}{}

\subsection{Minimum and Maximum Principles}

Before we prove the principles, we need to prove the following lemmas.
\begin{lem}\label{lem:oo} 
If $u$ is  $\opA$-superharmonic and $u=0$ a.e. in $\Omega,$  then $u\equiv 0$ in $\Omega.$
\end{lem}
\begin{proof}{It is enough to show that $u=0$ in a given ball $B\Subset\Omega.$   {By lower semicontinuity of $u$ infer that it is nonpositive.} By Lemma~\ref{lem:loc-bdd-superharm-are-supersol}, we get that $u\in W^{1,\vp(\cdot)}(\Omega).$ Let $v=P(u,B)$ be the Poisson modification of $u$ in $B.$ By Theorem~\ref{theo:Pois} we have that $v$ is continuous in $B$ and $v\leq u\leq 0.$ Therefore $v$ is an $\opA$-supersolution in $\Omega$ and $(u-v)\in W^{1,\vp(\cdot)}_0(\Omega)$. Moreover,
\[c_2^\opA \int_\Omega \vp(x,|Dv|)\,dx\leq \int_\Omega \opA(x,Dv)\cdot Dv\,dx\leq \int_\Omega \opA(x,Dv)\cdot Du\,dx=0,\]
where the last equality holds because $Du=0$ a.e. in $\Omega.$ But then, we directly get that $Dv=0$ and $v=0$ a.e. in $\Omega.$ By continuity of $v$ in $B$ we get that $v=0$ everywhere in $B$. In the view of $v\leq u\leq 0$, we get that also  $u\equiv 0$ in $\Omega.$}\end{proof}

\begin{lem}\label{lem:A-sh-lsc} If $u$ is $\opA$-superharmonic and finite a.e. in $\Omega$, then for every $x\in\Omega$ it holds that $u(x)=\liminf_{y\to x}u(y)={\rm ess}\liminf_{y\to x} u(y).$
\end{lem}
\begin{proof}
We fix arbitrary $x\in\Omega$ and by lower semicontinuity $u(x)\leq \liminf_{y\to x}u(y)\leq {\rm ess}\liminf_{y\to x} u(y)=:a.$ Let $\ve\in (0,a)$ and $B=B(x,r)\subset\Omega$ be such that $u(y)>a-\ve$ for a.e. $y\in B.$ By Corollary~\ref{coro:min-A-super} function $v=\min\{u-a+\ve,0\}$ is $\opA$-superharmonic in $\Omega$ and $v=0$ a.e. in $B.$ By Lemma~\ref{lem:oo} $v\equiv 0$ in $\Omega,$ but then $u(x)\geq a-\ve$. Letting $\ve\to 0$ we obtain that $u(x)=a$ and the claim is proven. 
\end{proof}

 {We define  $\psi:\Omega\times\rp\to\rp$ is given by \begin{equation}
    \label{psi}
\psi(x,s)=\vp(x,s)/s.
\end{equation}
Note that within our regime $s\mapsto\psi(\cdot,s)$ is strictly increasing, but not necessarily convex. Although in general $\psi$ does not generate the Musielak-Orlicz space, we still can define $\vr_{\psi(\cdot),\Omega}$ by~\eqref{modular} useful in quantifying the uniform estimates for trucations in the following lemma.

\begin{lem}\label{lem:unif-int} If for {$u$} there exist $M,k_0>0$, such that for all $k>k_0$\begin{equation}
    \label{apriori}\vr_{\vp(\cdot),B}(DT_k u)\leq Mk,
\end{equation} then there exists a function $\zeta:[0,|B|]\to\rp$, such that $\lim_{s\to 0^+}\zeta(s)=0$ and for every measurable set $E\subset B$ it holds that for all $k>0$ \[\vr_{\psi(\cdot),E}(D T_k u)\leq \zeta(|E|).\]
\end{lem}
\begin{proof}  The result is classical when $p=q$, \cite{hekima}. 
Therefore, we present the proof only for $p<q$. We start with observing that
\begin{flalign*}|\{x\in B \colon\, \vp(x,|Du|)>s\}|&\leq |\{x\in B \colon\,|u|>k\}|+ |\{x\in B \colon\,\vp(x,|Du|)>s,\ |u|\leq k\}|\\&=I_1+I_2.
\end{flalign*}
Let us first estimate the volume of superlevel sets of $u$ using Tchebyszev inequality, Poincar\'e inequality, assumptions on the growth of $\vp$, and~\eqref{apriori}. For all sufficiently large $k$ we have 
\begin{flalign*}
I_1&=|\{{x\in B \colon} |u|>k\}|\leq \int_B\frac{|T_k u|^p}{k^p}\,dx\leq\frac{c}{k^p} \int_B |DT_k u|^p\,dx\\
&\leq\frac{c}{k^p} \int_B \vp(x,|DT_k u|) \,dx= ck^{-p}\vr_{\vp(\cdot),B}(DT_k u)\leq cMk^{1-p}.
\end{flalign*}
Similarly by Tchebyszev inequality  and~\eqref{apriori} we can estimate also  
\begin{flalign*}
I_2=|\{x\in B \colon\ \vp(x,|Du|)>s,\ |u|\leq k\}|&\leq \frac{1}{s}\int_{\{\vp(x,DT_k u)>s\}}\vp(x,D T_k u)\,dx\leq  {M} \frac{k}{s}.
\end{flalign*}
Altogether for all sufficiently large $s$ (i.e. $s>k_0^p$) we have that
\begin{flalign*}
|\{x\in B \colon\ \vp(x,|Du|)>s\}|&\leq I_1+I_2\leq cs^{\frac{1-p}{p}}.
\end{flalign*}
Recall that due to~\eqref{doubl-star} there exists $C>0$  uniform in $x$  such that $\psi(x,s)\geq C \wt\vp^{-1}(x,\vp(x,s)),$ so
\begin{flalign*}
|\{x\in B \colon\,\psi(x,|Du|)>s\}|&\leq  |\{x\in B \colon\,C\wt\vp^{-1}(x,\vp(x,|Du|))>s\}|\\
&= |\{x\in B \colon\,\vp(x,|Du|)>\wt\vp(x,s/C)\}|\\
& \leq  |\{x\in B \colon\,\vp(x,|Du|)>(s/C)^{q'}\}|\leq c s^{-\frac{q'}{p'}}\,,
\end{flalign*}
for some $c>0$ independent of $x$. Since the case $q=p$ is trivial for these estimates, it suffices to consider $q>p$. Then ${-\frac{q'}{p'}}<{-1}$ and we get the uniform integrability of $\{\psi(\cdot,|DT_ku|)\}_k$, thus the claim follows.
\end{proof}
Let us sum up the information on integrability of gradients of truncations of $\opA$-superharmonic functions.
\begin{rem}\label{rem:unif-int}  For a function $u$ being $\opA$-superharmonic and finite a.e. in $\Omega$, by Lemma~\ref{lem:loc-bdd-superharm-are-supersol} we get that $\{T_k u\}$ is a sequence of $\opA$-supersolutions in $\Omega$. Then~\eqref{apriori} is satisfied because of the Caccioppoli estimate from Lemma~\ref{lem:A-supers-cacc}. Having Lemma~\ref{lem:unif-int} we get that there exists $R_0>0$, such that for every $x\in \Omega$ and $B=B(x,R)\Subset\Omega$ with $R<R_0$ we have   $\vr_{\psi(\cdot),B}(DT_k u)\leq 1$ for all $k>0$ and in fact also $\vr_{\psi(\cdot),B}(Du)\leq 1$ (where `$D$' stands for the generalized gradient, cf.~\eqref{gengrad}).
\end{rem}
\begin{lem}\label{lem:wH-for-trunc}
For $u$ being a nonnegative function  $\opA$-superharmonic and finite a.e. in $\Omega$ there exist constants $R^\opA_0=R_0^\opA(n)>0$, $s_0=s_0(\data,n)>0$ as in the weak Harnack inequality (Proposition~\ref{prop:weak-Har-super-inf}), and $C=C(\data,n)>0$, such that for every $k>1$ we have
\begin{equation}
    \label{in:wH-for-trunc}
\left(\barint_{B_R}(T_k u)^{s_0}\,dx\right)^\frac{1}{s_0} \leq C\left({\inf}_{B_{R/2}} (T_k u)+R\right)
\end{equation} 
for all $R\in(0,R_0^\opA]$ provided $B_{3R}\Subset\Omega$ and $\vr_{\psi(\cdot),B_{3R}}(Du)\leq 1.$
\end{lem}
\begin{proof} The proof is based on  Remark~\ref{rem:unif-int} and
Proposition~\ref{prop:weak-Har-super-inf} that provides weak Harnack inequality for an~$\opA$-supersolution $v$ holding with constant $C=C(\data,n)$ and for balls with radius $R<R_0(n)$ and so small that $\vr_{\vp(\cdot),B_{3R_0}}(Dv)\leq 1$. 

The only explanation is required whenever $|Dv|\geq 1$ a.e. in the considered ball. Then for every $k>1$ there exists $R_1(k)$ such that we get~\eqref{in:wH-for-trunc} for $T_k v$ over balls such that $R<\min\{R_1(k),R_0(n)\}$ and $\vr_{\vp(\cdot),B_{3R_1(k)}}(D T_kv)\leq 1$. Of course, then there exists $R_0^\opA(k)\in(0,R_1(k))$, such that we have~\eqref{in:wH-for-trunc} for $R<\min\{R_1(k),R_0(n)\}$ and $ \vr_{\psi(\cdot),B_{3R_0^\opA(k)}}(D T_kv)\leq \vr_{\vp(\cdot),B_{3R_0^\opA(k)}}(D T_kv)\leq 1$. Note that it is Remark~\ref{rem:unif-int} that allows us to choose $R_0^\opA$ independently of $k$.
\end{proof}}

We are in a position to prove that an $\opA$-harmonic function cannot attain its minimum nor maximum in a domain.

\begin{theo}[Strong Minimum Principle for $\opA$-superharmonic functions]\label{theo:mini-princ}  Suppose $u$ is $\opA$-superharmonic and finite a.e.  in connected set $\Omega$. If $u$ attains its minimum inside $\Omega,$ then $u$ is a constant function.
\end{theo} 
 
\begin{proof} {We consider $v=(u-\inf_{\Omega}u)$, which by Corollary~\ref{coro:min-A-super} is $\opA$-superharmonic. Let $E=\{x\in \Omega:\ v(x)= 0\}$, which by lower semicontinuity of $v$ (Lemma~\ref{lem:A-sh-lsc}) is nonempty and relatively closed in $\Omega.$ Having in hand Remark~\ref{rem:unif-int} we can choose $B=B(x,R)\subset 3B\Subset\Omega$ with radius smaller than $R_0^\opA$ from Lemma~\ref{lem:wH-for-trunc} and such that $\vr_{\psi(\cdot), B_{3R}}(Du)\leq 1$ where $\psi$ is as in~\eqref{psi}. Therefore, in the rest of the proof we restrict ourselves to a ball $B$. By Corollary~\ref{coro:min-A-super} functions $v$ and $T_k v$ are $\opA$-superharmonic in $3B$. Moreover, by Lemma~\ref{lem:loc-bdd-superharm-are-supersol} we infer that $\{T_k v\}$ is a~sequence of $\opA$-supersolutions integrable uniformly in the sense of Lemma~\ref{lem:unif-int}. We take any  $y\in B$ -- a Lebesgue's point of $T_k v$ for every $k$ and choose $B'=B'(y,R')\Subset B.$  Let us also fix arbitrary $k>0$. We have the weak Harnack inequality from Lemma~\ref{lem:wH-for-trunc} for $T_k v$ on $B'$ yielding
\[0\leq\left(\barint_{B'} (T_kv)^{s_0}\,dx\right)^\frac{1}{s_0}\leq C(\inf_{B'/2}T_k v+R')=CR'\]
with $s_0,C>0$ independent of $k$.  Letting $R'\to 0$ we get that $T_k v(y)=0$. Lebesgue's points of $T_k v$ for every $k$ are dense in $B$, we get that $T_k v\equiv 0$ a.e. in $B$. By arguments as in Lemma~\ref{lem:oo} we get that $T_k v\equiv 0$ in $B$, but then $B\subset E$ and $E$ has to be an open set. Since $\Omega$ is connected, $E$ is the only nonempty and relatively closed open set in $\Omega,$ that is $E=\Omega$. Therefore $T_kv\equiv 0$ in $\Omega.$ As $k>0$ was arbitrary $v=u-\inf_{\Omega}u\equiv 0$ in $\Omega$ as well.}
\end{proof}

The classical consequence of Strong Minimum Principle, we get its weaker form.

\begin{coro}[Minimum Principle for $\opA$-superharmonic functions]\label{coro:mini-princ} Suppose $u$ is $\opA$-superharmonic and finite a.e.  in $\Omega$. If $E\Subset\Omega$ a connected open subset of $\Omega$, then
\[\inf_{E} u=\inf_{\partial E} u.\]
\end{coro}

By the  very definition of an $\opA$-subharmonic function one gets the following direct consequence of the above fact.
\begin{coro}[Maximum Principle for $\opA$-subharmonic functions]\label{coro:max-princ}
Suppose $u$ is $\opA$-subharmonic and finite a.e. in $\Omega$. If $E\Subset\Omega$ a connected open subset of $\Omega$, then
\[\sup_E u=\sup_{\partial E} u.\]
\end{coro} 

Having Theorem~\ref{theo:mini-princ} and Corollary~\ref{coro:max-princ}, we infer that if $u$ is $\opA$-harmonic in $\Omega$, then it attains its minimum and maximum on $\partial\Omega$. In other words $\opA$-harmonic functions have the following Liouville-type property.
\begin{coro}[Liouville Theorem for $\opA$-harmonic functions]\label{coro:min-max-princ}
If an  $\opA$-harmonic function attains its extremum inside a domain, then it is a constant function. 
\end{coro}

\subsection{Boundary Harnack inequality for $\opA$-harmonic functions}

\begin{theo}[Boundary Harnack inequality for $\opA$-harmonic functions]\label{theo:boundary-harnack}
For a~nonnegative function $u$ which is $\opA$-harmonic in a connected set $\Omega$ there exist  $R_0=R(n)>0$ and $C=C(\data,n,R_0,{\rm ess\,sup}_{B_{R_0}}u)>0
$, such that 
\begin{equation*}
\sup_{\partial B_R} u\leq C(\inf_{\partial B_{R}} u+R)  \end{equation*} 
for all $R\in(0,R_0]$ provided $B_{3R}\Subset\Omega$ and $\vr_{\psi(\cdot),B_{3R}}(Du)\leq 1,$ where $\psi$ is given by~\eqref{psi}.
\end{theo}\begin{proof}It suffices to note that by Lemma~\ref{lem:A-h-is-great} we can use Minimum Principle of~Corollary~\ref{coro:mini-princ} and Maximum Principle of Corollary~\ref{coro:max-princ}. Then by Harnack inequality of Theorem~\ref{theo:Har-A-harm} the proof is complete.
\end{proof}

{\begin{coro}
Suppose $u$ is $\opA$-harmonic in $B_{\frac{3}{2}R}\setminus B_R,$ with $R<R_0 $ from Theorem~\ref{theo:boundary-harnack}, then exists $C=C(\data,n,R_0,{\rm ess\,sup}_{B_{R_0}}u)>0
$, such that
\begin{equation*}
\sup_{\partial B_{\frac 43 R}} u\leq C(\inf_{\partial B_{\frac 43 R}} u+2R).  \end{equation*} 
\end{coro}\begin{proof} Fix $\ve>0$ small enough for $B_R\Subset B_{\frac 43 R-\ve}\subset B_{\frac 43 R+\ve}\Subset B_{\frac 32 R}.$ Of course, then $u$ is $\opA$-harmonic in ${B_{\frac 43 R+\ve}\setminus B_{\frac 43 R-\ve}}$. We cover the annulus with finite number of~balls of equal radius as prescribed in the theorem and such that $\vr_{\psi(\cdot),B}(Du)\leq 1$, which is possible due to Remark~\ref{rem:unif-int}. Let us observe that due to the Harnack's inequality from Theorem~\ref{theo:boundary-harnack} we have\begin{flalign*}
\sup_{\partial B_{\frac 43 R+\ve}} u\leq \sup_{\partial B_{\frac 43 R+\ve}\cup\partial B_{\frac 43 R-\ve}} u&\leq C\Big( \inf_{\partial B_{\frac 43 R+\ve}\cup\partial B_{\frac 43 R-\ve}} u+ \frac 43 R+\ve\Big)\\&\leq C\Big(  \inf_{\partial B_{\frac 43 R+\ve}} u + 2R\Big).
\end{flalign*}
Since $u$ is continuous in $B_{\frac{3}{2}R}\setminus B_R,$ passing with $\ve\to 0$ we get the claim.
\end{proof}
}


\section*{Acknowledgements} I. Chlebicka is supported by NCN grant no. 2019/34/E/ST1/00120. A. Zatorska-Goldstein is supported by NCN grant no. 2019/33/B/ST1/00535.

\end{document}